\documentclass[11pt]{article} \usepackage{amsmath,amssymb,amsthm} \usepackage[unicode,breaklinks=true,colorlinks=true,linkcolor=blue,urlcolor=blue,citecolor=blue]{hyperref} \usepackage[dvipsnames]{xcolor} \usepackage[normalem]{ulem}  \usepackage{tikz} \usetikzlibrary{patterns} \usepackage{soul} \usepackage{enumerate} \usepackage[top=1in, bottom=1in, left=1.25in, right=1.25in, marginparwidth=1.1in, marginparsep=0.05in]{geometry} \numberwithin{equation}{section} \newtheorem{theorem}{Theorem}[section]  \newtheorem{lemma}[theorem]{Lemma} \newtheorem{definition}[theorem]{Definition}     \theoremstyle{remark} \newtheorem{remark}[theorem]{Remark}          \newcommand{\bket}[1]{\left\{ #1 \right\}} \def\llobet#1{} \newcommand{\norm}[1]{\| #1 \|} \def\labell#1{}  \newcommand{\bka}[1]{\langle #1 \rangle}  \newcommand{\al}{\alpha} \newcommand{\be}{\beta} \newcommand{\de}{\delta} \newcommand{\e}{\epsilon}    \newcommand{\la}{\lambda}  \newcommand{\Om}{{\Omega}}  \newcommand{\si}{\sigma} \newcommand{\td}{\tilde}   \newcommand{\MM}{M^{2,2}_{\mathcal C}} \newcommand{\MC}{\mathring M^{2,2}_{\mathcal C}}       \newcommand{\R}{{\mathbb R }} \newcommand{\N}{{\mathbb N}} \newcommand{\Z}{{\mathbb Z}}  \newcommand{\pd}{{\partial}} \newcommand{\nb}{{\nabla}} \newcommand{\lec}{\lesssim}  \newcommand{\I}{\infty}     \newcommand{\pv}{\mathop{\rm p.v.}}   \newcommand{\donothing}[1]{{}}  \newcommand{\EQ}[1]{\begin{equation}\begin{split} #1 \end{split}\end{equation}} \newcommand{\EQN}[1]{\begin{equation*}\begin{split} #1 \end{split}\end{equation*}} \DeclareMathOperator*{\esssup}{ess\,sup}  \def\indeq{\quad{}}  \def\colb{\color{black}} \definecolor{colorgggg}{rgb}{0.1,0.5,0.3} \definecolor{colorllll}{rgb}{0.0,0.7,0.0} \definecolor{colorhhhh}{rgb}{0,0.8,0.5} \definecolor{colorpppp}{rgb}{0.3,0.0,0.7}     \def\comma{ {\rm ,\qquad{}} }                                    \makeatletter \newcommand{\xRightarrow}[2][]{\ext@arrow 0359\Rightarrowfill@{#1}{#2}} \makeatother \newcommand{\loc}{\mathrm{loc}}  \newcommand{\uloc}{\mathrm{uloc}} \usepackage{accents} \newcommand*{\dt}[1] {\accentset{\mbox{\large\bfseries .}}{#1}}   \begin{document}   \title{Existence of global weak solutions to the Navier-Stokes equations in weighted spaces}    \author{Zachary Bradshaw, Igor Kukavica, and Tai-Peng Tsai}    \date{\today}   \maketitle  \small \noindent Department of Mathematics, University of Arkansas, Fayetteville, AR 72701\\e-mail: zb002@uark.edu\\ \noindent Department of Mathematics, University of Southern California, Los Angeles, CA 90089\\e-mail: kukavica@usc.edu\\ \noindent Department of Mathematics, University of British Columbia, Vancouver, BC V6T 1Z2\\e-mail: ttsai@math.ubc.ca\\ \medskip \normalsize \begin{abstract}  We obtain a global existence result for the three-dimensional Navier-Stokes equations with a large class of data  allowing growth at spatial infinity.  Namely, we show the global existence of suitable weak solutions when the initial data belongs to the weighted space $\mathring M^{2,2}_{\mathcal C}$  introduced in \cite{BK1}. This class  is strictly larger than currently available spaces of initial data for global existence and  includes all locally square integrable discretely self-similar data. We also identify a sub-class of data for which solutions exhibit eventual regularity on a parabolic set in space-time. \end{abstract} \section{Introduction}   The Navier-Stokes equations describe the evolution of the velocity~$u$ and the pressure~$p$, solving   \EQ{\label{eq.NSE}    &\partial_t u-\Delta u +u\cdot\nabla u+\nabla p = 0,    \\& \nabla \cdot u=0,    } in the sense of distributions~\cite{CF,G,RRS,LR,LR2,T,Tsai}. The system \eqref{eq.NSE} is set on $\R^3\times (0,T)$ where $T>0$ can be~$+\I$.  Also, $u$ evolves from a prescribed, divergence-free initial data $u_0\colon\R^3\to \R^3$.    \par \par In  \cite{leray}, \llobet{8ThswELzXU3X7Ebd1KdZ7v1rN3GiirRXGKWK099ovBM0FDJCvkopYNQ2aN94Z7k0UnUKamE3OjU8DFYFFokbSI2J9V9gVlM8ALWThDPnPu3EL7HPD2VDaZTggzcCCmbvc70qqPcC9mt60ogcrTiA3} J.~Leray constructed a global-in-time weak solution to \eqref{eq.NSE} on $\R^3\times (0,\infty)$ for any divergence-free vector field $u_0\in L^2(\R^3)$.  In \cite{LR}, Lemari\'e-Rieusset introduced a local analogue of a Leray weak solution evolving from uniformly locally square integrable \llobet{HEjwTK8ymKeuJMc4q6dVz200XnYUtLR9GYjPXvFOVr6W1zUK1WbPToaWJJuKnxBLnd0ftDEbMmj4loHYyhZyMjM91zQS4p7z8eKa9h0JrbacekcirexG0z4n3xz0QOWSvFj3jLhWXUIU21iIAwJtI} data  $u_0\in L^2_{\uloc}$. Here, $L^q_{\uloc}$, for $1\le q \le \infty$,  is the space of \llobet{3RbWa90I7rzAIqI3UElUJG7tLtUXzw4KQNETvXzqWaujEMenYlNIzLGxgB3AuJ86VS6RcPJ8OXWw8imtcKZEzHop84G1gSAs0PCowMI2fLKTdD60ynHg7lkNFjJLqOoQvfkfZBNG3o1DgCn9hyUh5} functions on $\R^3$ such that \[ \norm{u_0}_{L^q_{\uloc}} =\sup_{x \in\R^3} \norm{u_0}_{L^q(B(x,1))}<\infty. \] We also denote \[ E^q = \mathop{\rm Cl}\nolimits_{L^q_{\uloc}}({C_0^\I(\R^3)}) \] the closure of $C_0^\I(\R^3)$ in the $L^q_{\uloc}$-norm. We do not define solutions exactly as in \cite{LR}  but instead recall a definition from \cite{BK1}.     For a cube $Q$ in $\R^3$, we denote by $Q^*$ and $Q^{**}$ concentric cubes \llobet{VSP5z61qvQwceUdVJJsBvXDG4ELHQHIaPTbMTrsLsmtXGyOB7p2Os43USbq5ik4Lin769OTkUxmpI8uGYnfBKbYI9AQzCFw3h0geJftZZKU74rYleajmkmZJdiTGHOOaSt1NnlB7Y7h0yoWJryrVr} with side-lengths $4|Q|^{1/3}/3$, and $5|Q|^{1/3}/3$, respectively. Thus $Q \subset Q^* \subset Q^{**}$. This is a slightly different definition from \cite{BK1}. \par \begin{definition}[Local energy solutions]\label{def:localenergy} A vector field $u\in L^2_{\loc}(\R^3\times [0,T))$, where $0<T<\infty$, is a local energy solution to \eqref{eq.NSE} with divergence-free initial data $u_0\in L^2_{\loc}(\R^3)$   if the following conditions hold: \begin{enumerate} \item $u\in \bigcap_{R>0} L^{\I}(0,T;L^2(B_R(0)))$ and $\nabla u \in L^2_\loc(\R^3\times [0,T])$, \item for some $p\in L^{3/2}_{\loc}(\R^3\times (0,T))$, the pair $(u,p)$ is a distributional solution to \eqref{eq.NSE}, \item for all compact subsets $K$ of $\R^3$ we have $u(t)\to u_0$ in $L^2(K)$ \llobet{TzHO82S7oubQAWx9dz2XYWBe5Kf3ALsUFvqgtM2O2IdimrjZ7RN284KGYtrVaWW4nTZXVbRVoQ77hVLX6K2kqFWFmaZnsF9Chp8KxrscSGPiStVXBJ3xZcD5IP4Fu9LcdTR2VwbcLDlGK1ro3EEyq} as $t\to 0^+$, \item $u$ is suitable in the sense of Caffarelli-Kohn-Nirenberg, i.e.,  for all non-negative $\phi\in C_0^\infty ({\mathbb R}^{3}\times (0,T))$,  we have  the local energy inequality \EQ{\label{CKN-LEI} & 2\iint |\nabla u|^2\phi\,dx\,dt  \leq  \iint |u|^2(\partial_t \phi + \Delta\phi )\,dx\,dt +\iint (|u|^2+2p)(u\cdot \nabla\phi)\,dx\,dt, } \item the function $t\mapsto \int u(x,t)\cdot w(x)\,dx$ is continuous on $[0,T)$ \llobet{EAzw6sKeEg2sFfjzMtrZ9kbdxNw66cxftlzDGZhxQAWQKkSXjqmmrEpNuG6Pyloq8hHlSfMaLXm5RzEXW4Y1Bqib3UOhYw95h6f6o8kw6frZwg6fIyXPnae1TQJMt2TTfWWfjJrXilpYGrUlQ4uM7} for any compactly supported $w\in L^2(\R^3)$, \item for every cube $Q\subset \R^3$, there exists $p_{Q}(t)\in L^{3/2}(0,T)$ such that for $x\in Q^*$ and $0<t<T$,  \EQN{ p(x,t)-p_Q(t) &=   - \frac 13  |u(x)|^2 +\pv\int_{ y\in Q^{**}}  K_{ij}(x-y)   (u_i(y,s)u_j(y,s))\,dy    \\&\indeq  + \int_{ y\notin Q^{**}}  ( K_{ij}(x-y) - K_{ij}(x_Q - y)  )  (u_i(y,s)u_j(y,s)\,dy, } where $x_Q$ is the center of $Q$ and $K_{ij}(y)=\partial_i\partial_j (4\pi |y|)^{-1}$.  \end{enumerate} We say that $u$ is a local energy solution on $\R^3\times [0,\I)$ if it is a local energy solution on $\R^3\times [0,T)$ for all $T<\I$. \end{definition} \par \par In comparison with the definitions in  \cite{LR,KiSe,RS,JiaSverak-minimal,BT8}, we   do not require $u_0\in L^2_\uloc$ or \begin{equation}\notag \esssup_{0\leq t<R^2}\,\sup_{x_0\in \R^3}\, \int_{B_R(x_0 )}|u(x,t)|^2\,dx + \sup_{x_0\in \R^3}\int_0^{R^2}\int_{B_R(x_0)} |\nabla u(x,t)|^2\,dx \,dt<\infty , \end{equation} for any $R>0$. For other low-regularity solution classes cf.~\cite{CW, KV}. Suitability is an \llobet{Dsp0rVg3gIEmQOzTFh9LAKO8csQu6mh25r8WqRIDZWgSYkWDulL8GptZW10GdSYFUXLzyQZhVZMn9amP9aEWzkau06dZghMym3RjfdePGln8s7xHYCIV9HwKa6vEjH5J8Ipr7NkCxWR84TWnqs0fs} important property as it allows the application of the Caffarelli-Kohn-Nirenberg \cite{CKN} and related regularity criteria and can be used to establish certain turbulent dynamics \cite{DaGr}. \par Local energy solutions are known to exist locally in time for initial data in $L^2_\uloc(\Om)$ where $\Om$ is $\R^3$ \cite{LR,LR2,KiSe,KwTs} or $\R^3_+$ \cite{MaMiPr}.  The global existence usually requires some type of decay or a structural assumption on the data~\cite{LR,KiSe,KwTs,BT5,CW,BT8}. In particular, in \cite{BT8} two of the authors constructed global solutions for initial data satisfying   \[    \lim_{R\to \I} \sup_{x_0\in \R^3} \frac 1 {R^2} \int_{B_R(x_0)} |u_0|^2\,dx =0.   \] In two dimensions, Basson constructed local solutions for a new class of (specified below) and global solutions  for non-decaying data in the uniform space $L^2_\uloc$ \cite{Basson}. \par Recently, several  existence results have appeared which allow for initial data in $L^2_\loc\setminus L^2_\uloc$.  In \cite{BK1}, two of the authors constructed local in time solutions for initial data  in the space $\mathring M^{2,2}_\mathcal C$ defined below. This construction is related to the idea of Basson~\cite{Basson}, but is given in three dimensions as opposed to two, in a different  setting, and uses a different approach to the pressure and estimates. In \cite{FL}, Fern\'andez-Dalgo and Lemari\'e-Rieusset constructed global solutions assuming the initial data $u_0$ satisfies  \[ \int_{\R^3} \frac {|u_0(x)|^2} {(1+|x|)^2} \,dx <\I.  \] Fern\'andez-Dalgo and Lemari\'e-Rieusset also included a new construction of discretely \llobet{iPqGgsId1fs53AT71qRIczPX77Si23GirL9MQZ4FpigdruNYth1K4MZilvrRk6B4W5B8Id3Xq9nhxEN4P6ipZla2UQQx8mdag7rVD3zdDrhBvkLDJotKyV5IrmyJR5etxS1cvEsYxGzj2TrfSRmyZ} self-similar (DSS) solutions in \cite{FL} for any DSS data in $L^2_\loc$ based on their global existence result (this gives a new proof of results in \cite{CW,BT5}; cf.~also~\cite{JiaSverak,BT1}) and is therefore of interest beyond the existence problem. \par The purpose of this paper is to construct global solutions for the class of initial data  $\mathring M^{2,2}_\mathcal C$ introduced in \cite{BK1}, which strictly includes the initial data from \cite{BT8,FL}. For $n\in{\mathbb N}_0$, let    \EQ{       Q_{n}=\{x: |x_i|< 2^{n+1}\hbox{~for~}i=1,2,3\}.          \label{EQL5TS2kIQjZKb9QUx2Ui5Aflw1SLDGIuWUdCPjywVVM2ct8cmgOBS7dQViXR8Fbta1mtEFjTO0kowcK2d6MZiW8PrKPI1sXWJNBcREVY4H5QQGHbplPbwdTxpOI5OQZAKyiix7QeyYI91Ea16rKXKL2ifQXQPdPNL605}   } Denote  \[ S_0=\overline{Q_0},     \quad S_n=\overline{Q_{n}\setminus Q_{n-1}} \quad \text{for} \quad n\in \mathbb N. \]    Partition $S_0$ into $64$ cubes of side-length $1$ and $S_n$ into $56$ cubes of side-length $2^n$.  Let $\mathcal C$ be the collection of these cubes.  Note that the number of cubes in  $\bigcup_{i=0}^{n-1} S_i$ grows {linearly} in~$n$. \par The main features of the collection $\mathcal C$ are the following:    \begin{enumerate}[(i)]      \item The \labell{EJiHcK} side-length \labell{rBs2qGtQbaqedOjLixjGiNWr1PbYSZeSxxFinaK9EkiCHV2a13f7G3G3oDKK0ibKVy453E2nFQS8Hnqg0E32ADddEVnmJ7HBc1t2K2ihCzZuy9kpsHn8KouARkvsHKPy8YodOOqBihF1Z3CvUFhmjgBmuZq7ggWL} of a cube       is proportional \labell{g5dQB1kpFxkk35GFodk00YD13qIqqbLwyQCcyZRwHAfp79oimtCc5CV8cEuwUw7k8Q7nCqWkMgYrtVRIySMtZUGCHXV9mr9GHZol0VEeIjQvwgw17pDhXJSFUcYbqUgnGV8IFWbS1GXaz0ZTt81w7EnIhFF72v2P} to the distance of its center   \labell{kWOXlkrw6IPu5679vcW1f6z99lM2LI1Y6Naaxfl18gT0gDptVlCN4jfGSbCro5Dv78CxaukYiUIWWyYDRw8z7KjPx7ChC7zJvb1b0rFd7nMxk091wHvy4u5vLLsJ8NmAkWtxuf4P5NwP23b06sFNQ6xgDhuRGbK7}    from the origin.      \item Adjacent cubes have comparable volumes.      \item If $|Q'|<|Q|$, then the distance between the centers of $Q$ and $Q'$ is proportional to $|Q|^{1/3}$.      \item The number of cubes $Q'$ satisfying $|Q'|<|Q|$ is bounded above by a constant multiple of             $ \log|Q|$.    \end{enumerate} \par For convenience we also refer to the \labell{j2O4gy4p4BLtop3h2kfyI9wO4AaEWb36YyHYiI1S3COJ7aN1r0sQOrCAC4vL7yrCGkIRlNuGbOuuk1awLDK2zlKa40hyJnDV4iFxsqO001rqCeOAO2es7DRaCpUG54F2i97xSQrcbPZ6K8Kudn9e6SYo396Fr8LU} collection of cubes in $\mathcal C$ contained in $S_n$ as $S_n$ and accordingly write $Q\in S_n$  if $Q$ is part of the collection $S_n$. \par Our initial data space is an analogue of $L^2_\uloc$ but adapted to the cover $\mathcal C$ and weighted. \par    \begin{definition}\label{def.spaces}  Let $p\in[1,\infty)$ and $q\ge0$. We have $f\in M^{p,q}_{\mathcal C}$ if       \[      \|f\|_{M^{p,q}_{\mathcal C}}^p := \sup_{Q\in \mathcal C} \frac 1 {|Q|^{q/3}}\int_{Q} |f(x)|^p\,dx<\I.     \] Let $\mathring M^{p,q}_{\mathcal C}$ be the set of $f\in M^{p,q}_{\mathcal C}$ such that  \[ \frac 1{|Q|^{q/3}} \int_Q |f|^p \,dx \to 0 \text{ as } |Q|\to \I, Q\in \mathcal C. \] \end{definition}  \par \par Our main result asserts the existence of a  global in time local energy solutions  to the Navier-Stokes equations with data in $\mathring M^{2,2}_{\mathcal C}$. \par \par \par \begin{theorem}[Global existence]\label{thrm.main} Assume $u_0\in \mathring M^{2,2}_{\mathcal C}$ is divergence-free. Then there exists $u\colon\R^3\times (0,\I)\to \R^3$ and $p\colon\R^3\times (0,\I)\to \R$   so that $(u,p)$ is a local energy solution to the Navier-Stokes \llobet{BzQ3ETfDlCad7VfoMwPmngrDHPfZV0aYkOjrZUw799etoYuBMIC4ovEY8DOLNURVQ5lti1iSNZAdwWr6Q8oPFfae5lAR9gDRSiHOeJOWwxLv20GoMt2Hz7YcalyPZxeRuFM07gaV9UIz7S43k5TrZ} equations on $\R^3\times (0,\I)$. \end{theorem} \par \par Theorem~\ref{thrm.main} is proven  by  first establishing the equivalence of the norm in  $\mathring M^{2,q}_{\mathcal C}$ with the norm on the space $\mathring M^{2,q}_{\mathcal C_n}$ described next. Let   \[   \| u_0\|_{ M_{\mathcal C_n}^{2,q}}^2 = \sup_{Q\in \mathcal C_n} \frac 1 {|Q|^{q/3}} \int_Q |u_0|^2\,dx,  \] where   \begin{equation}    \mathcal C_n = \{ Q_{n-1}\} \cup \{ Q\in \mathcal C : \exists k\geq n \text{ such~that }Q\in  S_k   \},       \label{EQL5TS2kIQjZKb9QUx2Ui5Aflw1SLDGIuWUdCPjywVVM2ct8cmgOBS7dQViXR8Fbta1mtEFjTO0kowcK2d6MZiW8PrKPI1sXWJNBcREVY4H5QQGHbplPbwdTxpOI5OQZAKyiix7QeyYI91Ea16rKXKL2ifQXQPdPNL611}   \end{equation} and, by \eqref{EQL5TS2kIQjZKb9QUx2Ui5Aflw1SLDGIuWUdCPjywVVM2ct8cmgOBS7dQViXR8Fbta1mtEFjTO0kowcK2d6MZiW8PrKPI1sXWJNBcREVY4H5QQGHbplPbwdTxpOI5OQZAKyiix7QeyYI91Ea16rKXKL2ifQXQPdPNL605}, $Q_{n} = \bigcup_{i\leq n;Q'\in S_i} Q'$.   Thus we have replaced cubes in $\mathcal C$ with the side-lengths \llobet{iDMt7pENCYiuHL7gac7GqyN6Z1ux56YZh2dyJVx9MeUOMWBQfl0EmIc5Zryfy3irahCy9PiMJ7ofoOpdennsLixZxJtCjC9M71vO0fxiR51mFIBQRo1oWIq3gDPstD2ntfoX7YUoS5kGuVIGMcfHZ} less than $ 2^n$ by a single cube of side-length $2^{n+1}$.  We say $u_0\in M_{\mathcal C_n}^{2,q}$ if $\|u_0\|_{M_{\mathcal C_n}^{2,q}} <\I$. The relationship between  $\mathring M^{2,2}_{\mathcal C}$ and $\mathring M^{2,2}_{\mathcal C_n}$ is described in Lemma~\ref{lemma.init} \llobet{e37ZoGA1dDmkXO2KYRLpJjIIomM6Nuu8O0jO5NabUbRnZn15khG94S21V4Ip457ooaiPu2jhIzosWFDuO5HdGrdjvvtTLBjovLLiCo6L5LwaPmvD6Zpal69Ljn11reT2CPmvjrL3xHmDYKuv5TnpC} below. The solution constructed in the proof of  Theorem~\ref{thrm.main} also satisfies \[ \esssup_{0<s<t} \| u(s)\|_{  M^{2,2}_{\mathcal C_n}}^2+\sup_{Q\in \mathcal C_n} \frac 1 {|Q|^{2/3}}\int_0^t\int_{Q}|\nb u(x,s)|^2\,dx\,ds    \leq C\|u_0\|_{M^{2,2}_{\mathcal C_n}}^2 \] for $t\leq T_n$, \llobet{1fMoURRToLoilk0FEghakm5M9cOIPdQlGDLnXerCykJC10FHhvvnYaTGuqUrfTQPvwEqiHOvOhD6AnXuvGlzVAvpzdOk36ymyUoFbAcAABItOes52Vqd0Yc7U2gBt0WfFVQZhrJHrlBLdCx8IodWp} where $T_n$ is given in \eqref{EQL5TS2kIQjZKb9QUx2Ui5Aflw1SLDGIuWUdCPjywVVM2ct8cmgOBS7dQViXR8Fbta1mtEFjTO0kowcK2d6MZiW8PrKPI1sXWJNBcREVY4H5QQGHbplPbwdTxpOI5OQZAKyiix7QeyYI91Ea16rKXKL2ifQXQPdPNL607} below. \par Note that  Theorem~\ref{thrm.main}  improves the constructions in \cite{BK1,FL,BT8}. The details are given in Section~\ref{sec.data}. \par In \cite{Basson}, Basson considered local existence in two dimensions for initial data $u_0$ satisfying  \[ \sup_{R\geq 1} \frac 1 {R^2}\int_{|x|\leq R} |u_0|^2\,dx<\I. \] Basson left open the global \llobet{AlDS8CHBrNLzxWp6ypjuwWmgXtoy1vPbrauHyMNbkUrZD6Ee2fzIDtkZEtiLmgre1woDjuLBBSdasYVcFUhyViCxB15yLtqlqoUhgL3bZNYVkorzwa3650qWhF22epiXcAjA4ZV4bcXxuB3NQNp0G} existence of weak solutions in both~2D and~3D. In three dimensions  \[ u_0\in \mathring M^{2,2}_\mathcal C \iff \lim_{R\to \I} \frac 1 {R^2}\int_{|x|\leq R} |u_0|^2\,dx=0. \] Based on the global existence result Theorem~\ref{thrm.main},  we  expect that, in two dimensions, we get  \llobet{xW2Vs1zjtqe2pLEBiS30E0NKHgYN50vXaK6pNpwdBX2Yv7V0UddTcPidRNNCLG47Fc3PLBxK3Bex1XzyXcj0Z6aJk0HKuQnwdDhPQ1QrwA05v9c3pnzttztx2IirWCZBoS5xlOKCiD3WFh4dvCLQA} the global existence  if  \[ \lim_{R\to \I} \frac 1 {R^2}\int_{|x|\leq R} |u_0|^2\,dx=0, \] which is a subset of Basson's class that has very mild decay at spatial infinity.   \par \par \par \par Using the  new a priori bounds in the proof of Theorem~\ref{thrm.main} allows us to establish the following eventual regularity result. \par \begin{theorem}[Eventual Regularity]\label{thrm.er} Assume $u_0\in \mathring M^{2,1}_{\mathcal C}$  is divergence-free, and let $(u,p)$ be a local energy solution on $\R^3\times (0,\I)$ with initial data $u_0$. Assume additionally that   \EQN{ \esssup_{0<s<t} \| u(s)\|_{  M^{2,1}_{\mathcal C}}^2+\sup_{Q\in \mathcal C} \frac 1 {|Q|^{1/3}}\int_0^t\int_{Q}|\nb u(x,s)|^2\,dx\,ds<\I, } for all $t<\I$. Then for any $\delta>0$, there exists a time $\tau$ depending only on $u_0$ and $\delta$  so that $u$ is smooth on  \EQN{ \bigl\{  (x,t) : t> \max \{\delta|x|^2, \tau\}\bigr\}. } \end{theorem} \par  This follows from a~priori bounds below,  ideas in \cite{BT8}, and a  partial regularity result of~\cite{CKN} (in particular, the version in \cite{L98}); cf.~also \cite{Ku,LS99}.  This is variant of the usual notion of eventual regularity as it is not uniform in $x\in \R^3$.  The class $\mathring M^{2,1}_{\mathcal C}$ excludes DSS data and leaves DSS solutions as the borderline candidate for the failure of eventual regularity even for our non-uniform version.  Note that the shape of the regular set is consistent with that obtained for DSS solutions in~\cite{KaMiTs}. In contrast, all self-similar solutions in this class are \labell{xyXOjdFsMrl54EhT8vrxxF2phKPbszrlpMAubERMGQAaCBu2LqwGasprfIZOiKVVbuVae6abaufy9KcFk6cBlZ5rKUjhtWE1Cnt9RmdwhJRySGVSOVTv9FY4uzyAHSp6yT9s6R6oOi3aqZlL7bIvWZ18cFaiwptC} regular by \cite{grujic}.  \par It turns out that the spaces $M^{p,q}_\mathcal C$ are equivalent to certain Herz spaces. Let $A_k = \{ x \in \R^n: 2^{k-1} \le |x|< 2^k\}$. For $n \in \N$, $s \in \R$ and $p,q\in (0,\infty]$, the \emph{homogeneous Herz space} $\dt K^{s}_{p,q}(\R^n)$ is the space of functions $f \in L^p_{\loc}(\R^n \setminus \{0\})$ with finite norm \[ \norm{f}_{\dt K^{s}_{p,q}} = \left \{ \begin{alignedat}{3} &  \bigg( \sum_{k \in \Z}  2 ^{ksq} \norm{f}_{L^p(A_k)}^q \bigg)^{1/q}  \qquad &\text{if } q<\I,  \\ &  \sup _{k \in \Z} 2^{ks} \norm{f}_{L^p(A_k)}  \qquad &\text{if } q=\I. \end{alignedat} \right . \] The \emph{non-homogeneous Herz space} $K^{s}_{p,q}(\R^n)$ is, with $A_0$ redefined as $B_1$, \[ \norm{f}_{ K^{s}_{p,q}} = \left \{ \begin{alignedat}{3} &  \bigg( \sum_{k \in \N_0}  2 ^{ksq} \norm{f}_{L^p(A_k)}^q \bigg)^{1/q}  \qquad &\text{if } q<\I,  \\ &  \sup _{k \in \N_0} 2^{ks} \norm{f}_{L^p(A_k)}  \qquad &\text{if } q=\I. \end{alignedat} \right . \] Then the space $M{}^{p,q}_{\mathcal C}$ is \llobet{NAQJGgyvODNTDFKjMc0RJPm4HUSQkLnTQ4Y6CCMvNjARZblir7RFsINzHiJlcgfxSCHtsZOG1VuOzk5G1CLtmRYIeD35BBuxZJdYLOCwS9lokSNasDLj5h8yniu7hu3cdizYh1PdwEl3m8XtyXQRC} equivalent to the space with  the norm \[    \left( \int_{|x|<1} |f|^p dx \right)^{1/p}+           \sup_{R=2^k,k\in\N}         \left(\frac 1{R^q} \int_{R<|x|<2 R} |f|^p dx\right)^{1/p}    , \] i.e., with the  non-homogeneous Herz space $K^{-q}_{p,\infty}$. Local in time existence of mild solutions for large data in some subcritical weak Herz spaces (which include the Herz spaces) has been established by Tsutsui \cite{Tsutsui} ($p>3$ is required). These spaces also appear in \cite{KaMiTs} \par The paper is organized as follows.  In Section~\ref{sec.data} we reformulate the assumption $u_0\in \mathring M^{2,2}_\mathcal C$ in a more usable context.  In Section~\ref{sec.aprioribounds} we establish a new a priori bound that can be pushed to arbitrarily large times. In Section~\ref{sec.approximation} we use the a priori bound from Section~\ref{sec.aprioribounds} to construct global solutions, while in Section~\ref{sec.er} we prove Theorem~\ref{thrm.main}. \par \par \section{The space of initial data}\label{sec.data} \par We first slightly  reframe the assumption on the initial data.  \par \begin{lemma}\label{lemma.init} Let $f\in L^2_\loc$.  The following statements are equivalent \begin{enumerate} \item $f\in \mathring M^{2,2}_{{\mathcal C}}$, \item $\|f\|_{M^{2,2}_{\mathcal C_n}} \to 0$ as $n\to \I$, \item $\lim_{R\to \I} \frac 1 {R^2} \int_{B_R(0)} |f|^2\to 0$. \end{enumerate} \end{lemma} \par \begin{proof}  We first check that our decay condition (2) matches (3). Assume $\|f\|_{M^{2,2}_{\mathcal C_n} } \to 0$ as $n\to \I$.   This clearly implies $  R^{-2} \int_{B_R(0)}  |f|^2\,dx \to 0$. On the other hand, if $  R^{-2} \int_{B_R(0)}  |f|^2\,dx \to 0$, then, for $Q\in \mathcal C$ with    $|Q|$ sufficiently large, we can make sure that  \[ \frac 1 {|Q|^{2/3}} \int_{Q} |f|^2 \leq C \frac 1 {R^2} \int_{B_R(0)} |f|^2 \,dx \] is arbitrarily small. This also implies  \[ \frac 1 {|Q_{n-1}|^{2/3}} \int_{Q_{n-1}}|f|^2\,dx \to 0, \] as $n\to \I$.  Since $\mathcal C_n$ contains cubes with side-length larger than $2^n$ and $Q_{n-1}$, we may  take $n$ sufficiently large and ensure that $\|f\|_{M^{2,2}_{\mathcal C_n}}$ is small. Thus (2) and (3) are equivalent. \par Finally we check the equivalence of (1) and (2). Assume $f\in \mathring M^{2,2}_{{\mathcal C}}$. It suffices to show \[ \frac 1 {2^{2n}} \int_{Q_n} |f|^2\,dx\to 0. \] Let $\e>0$ be given. Then, \llobet{AbweaLiN8qA9N6DREwy6gZexsA4fGEKHKQPPPKMbksY1jM4h3JjgSUOnep1wRqNGAgrL4c18Wv4kchDgRx7GjjIBzcKQVf7gATrZxOy6FF7y93iuuAQt9TKRxS5GOTFGx4Xx1U3R4s7U1mpabpDHg} since $f\in \mathring M_{{\mathcal C}}^{2,2}$, for any constant $\bar C$ there exists $N$ so that  \[ \frac 1 {|Q|^{2/3}} \int_Q |f|^2 \,dx \leq \frac \e {2\bar C}, \] provided $|Q|\geq 2^{3N}$ and $Q\in \mathcal C$. We now have  \[ \frac 1 {2^{2n}} \int_{Q_n} |f|^2\,dx \leq \sum_{Q\in \mathcal C; Q\subset Q_n; |Q|\geq 2^{3N}} \frac 1 {2^{2n}} \int_Q |f|^2\,dx +\sum_{Q\in \mathcal C; Q\subset Q_n; |Q|< 2^{3N}} \frac 1 {2^{2n}} \int_Q |f|^2\,dx. \] By our choice of $N$ we have  \[ \sum_{Q\in \mathcal C; Q\subset Q_n; |Q|\geq 2^{3N}} \frac 1 {2^{2n}} \int_Q |f|^2\,dx \leq \frac \e {2\bar C} \sum_{Q\in \mathcal C;Q\subset Q_n} \frac {|Q|^{2/3}} {2^{2n}} \leq \frac \e 2, \] where we have set $\bar C$ to be the uniform (in $n$) bound on the partial sum appearing above.  On the other hand, \[ \sum_{Q\in \mathcal C; Q\subset Q_n; |Q|< 2^{3N}} \frac 1 {2^{2n}} \int_Q |f|^2\,dx \leq \frac 1 {2^{2n}}  \|f\|_{M^{2,2}_\mathcal C}^2 \sum_{Q\in \mathcal C;|Q|<2^{3N}} |Q|^{2/3}. \] The last sum is bounded by  $C 2^{2N}$. By requiring $n$ to be large in comparison to $N$, we can ensure the last expression is smaller than $\e/2$. Hence, for a sufficiently large $n$,  \[ \frac 1 {2^{2n}} \int_{Q_n} |f|^2\,dx<\e,\] and we have established that (1) implies (2). \par Conversely, if $\|f\|_{\mathcal C_n} \to 0 $, then $f\in \mathring M^{2,2}_{{\mathcal C}}$ follows immediately. \end{proof} \par \par \par  \section{A priori bounds}\label{sec.aprioribounds} \par \par \labell{1ndFyp4oKxDfQz28136a8zXwsGlYsh9Gp3TalnrRUKttBKeFr4543qU2hh3WbYw09g2WLIXzvQzMkj5f0xLseH9dscinGwuPJLP1gEN5WqYsSoWPeqjMimTybHjjcbn0NO5hzP9W40r2w77TAoz70N1au09bocDS} In this section we work exclusively with $\mathcal C_n$, where $n\in {\mathbb N}$ is given. Let $Q_n = \bigcup_{m\leq n;Q'\in S_m}  Q'$.  \par \subsection{Pressure estimate} \par \par We adapt the pressure \llobet{B7BgkPk6nJh01PCxdDsw514O648VD8iJ54FW6rs6SyqGzMKfXopoe4eo52UNB4Q8f8NUz8u2nGOAXHWgKtGAtGGJsbmz2qjvSvGBu5e4JgLAqrmgMmS08ZFsxQm28M3z4Ho1xxjj8UkbMbm8M0cLP} estimates  in \cite{BK1} to the  $M^{2,q}_{\mathcal C_n}$ framework.   Denote  \[  G_{ij}f(x)=R_i R_j f(x)= - \frac 13 \de_{ij} f(x) + \pv \int K_{ij}(x-y) f(y)\,dy,\]  where $R_i$ denotes the $i$-th Riesz transform and   \[  K_{ij}(y) =  \pd_i \pd_j \frac1{4\pi |y|} = \frac {-\de_{ij}|y|^2 + 3 y_i y_j}{4\pi |y|^5}.\] Fix a cube $Q \subset \R^3$, and let $Q^*,Q^{**}$ be as defined preceding Definition \ref{def:localenergy}.     For $x\in Q^*$, let  \EQ{ \label{EQL5TS2kIQjZKb9QUx2Ui5Aflw1SLDGIuWUdCPjywVVM2ct8cmgOBS7dQViXR8Fbta1mtEFjTO0kowcK2d6MZiW8PrKPI1sXWJNBcREVY4H5QQGHbplPbwdTxpOI5OQZAKyiix7QeyYI91Ea16rKXKL2ifQXQPdPNL601} G_{ij}^Q f(x) &=   - \frac 13 \de_{ij} f(x) +\pv \int_{ y\in Q^{**}}  K_{ij}(x-y)   f(y)\,dy \\&\indeq +  \int_{ y\notin Q^{**}}  \bigl( K_{ij}(x-y) - K_{ij}(x_Q - y  )\bigr)  f(y)\,dy.  } We say that a pressure $p$, associated to a local \llobet{kicxaCjkhnobr0p4codyxTCkVj8tW4iP2OhTRF6kU2k2ooZJFsqY4BFSNI3uW2fjOMFf7xJveilbUVTArCTvqWLivbRpg2wpAJOnlRUEPKhj9hdGM0MigcqQwkyunBJrTLDcPgnOSCHOsSgQsR35M} energy solution $u$, satisfies the pressure expansion if \EQN{ p(x,t)-p_Q(t)=   (G_{ij}^Q u_i u_j )(x,t) \comma x\in Q^{*}, } for some $p_{Q}(t)\in L^{3/2}(0,T)$. \par \par \par \par \begin{lemma}\label{lemma.pressure} Fix $n\in \mathbb N$. Assume $u$ is a local energy solution to \eqref{eq.NSE} on $\R^3\times [0,T]$ with an associated pressure $p$. Then, for   $Q\in \mathcal C_n$ and $t<T$, \EQ{ &\frac 1 {|Q|^{1/3}} \int_0^t \int_{Q^*} |p-p_Q(s) |^{3/2}\,dx\,ds \\& \indeq \leq  C \sup_{Q'\in \mathcal C_n ; Q'\cap Q^{**}\neq \emptyset} \frac 1 {|Q'|^{1/3}} \int_0^t \int_{Q'} |u|^3\,dx\,ds  \\& \indeq\indeq + C |Q|^{q/2-5/6} (\log \bka{|Q|^{1/3}/2^n})^{3/2 }\int_0^t  \bigg(\sup_{Q'\in \mathcal C_n}\frac 1 {|Q'|^{q/3}} \int_{Q'} |u(x,s)|^2\,dx    \bigg)^{3/2} \,ds,    \label{EQL5TS2kIQjZKb9QUx2Ui5Aflw1SLDGIuWUdCPjywVVM2ct8cmgOBS7dQViXR8Fbta1mtEFjTO0kowcK2d6MZiW8PrKPI1sXWJNBcREVY4H5QQGHbplPbwdTxpOI5OQZAKyiix7QeyYI91Ea16rKXKL2ifQXQPdPNL608} } where we denote $\bka{x} = (1+|x|^2)^{1/2}$. \end{lemma}  \par \begin{proof} The proof follows \cite[Proof of Lemma~2.1]{BK1} with an improved bound on the far part of the pressure (cf.~the last term in \eqref{EQL5TS2kIQjZKb9QUx2Ui5Aflw1SLDGIuWUdCPjywVVM2ct8cmgOBS7dQViXR8Fbta1mtEFjTO0kowcK2d6MZiW8PrKPI1sXWJNBcREVY4H5QQGHbplPbwdTxpOI5OQZAKyiix7QeyYI91Ea16rKXKL2ifQXQPdPNL608}). For $x\in Q^*$, write  \[p(x,t)-p_Q(t) = I_\text{near}(x,t)+ I_\text{far}(x,t),\]  where $I_\text{near} (x,t)$ is the sum of the first two terms on the right hand side of \eqref{EQL5TS2kIQjZKb9QUx2Ui5Aflw1SLDGIuWUdCPjywVVM2ct8cmgOBS7dQViXR8Fbta1mtEFjTO0kowcK2d6MZiW8PrKPI1sXWJNBcREVY4H5QQGHbplPbwdTxpOI5OQZAKyiix7QeyYI91Ea16rKXKL2ifQXQPdPNL601} and $I_\text{far}(x,t)$  is the last term where $f=u_i u_j$. \par The Calder\'on-Zygmund inequality implies \EQN{ \frac 1 {|Q|^{1/3}} \int_0^t \int_{Q^*} |I_\text{near}|^{3/2}\,dx\,ds   &\leq  \frac C {|Q|^{1/3}} \int_0^t \int_{Q^{**}} |u|^3 \,dx\,ds. } Since $|Q|\sim |Q'|$ whenever $Q'\in \mathcal C_n$ and $Q' \cap  Q^{**} \neq 0 $ and there are a bounded number of cubes satisfying this property (independent of $Q$ and $n$), we have    \begin{align}    \begin{split}      &      \frac 1 {|Q|^{1/3}} \int_0^t \int_{Q^{**}} |u|^3 \,dx\,ds       \leq       C \sup_{Q'\cap Q^{**}\neq \emptyset,   Q'\in{\mathcal C}_n}\frac 1 {|Q'|^{1/3}}\int_0^t \int_{Q'}|u|^3\,dx\,ds.       \end{split}     \notag   \end{align} \par We now estimate $I_\text{far}$.   Recall that if $x_Q$ is the center of $Q$, then $Q^*$ is the cube centered at $x_Q$ with the side-length $4|Q|^{1/3}/3$, and  $Q^{**}$ is the same but with the side-length  $5|Q|^{1/3}/3$.  Also, if $Q\in \mathcal C_n$, then $Q^*$ and $Q^{**}$ only overlap cubes in $\mathcal C_n$ which are adjacent to $Q$.  If $y\notin Q^{**}$ and $x\in Q^*$, then \[ |K_{ij}(x-y)-K_{ij}(x_Q-y) |\leq       \frac{       C |Q|^{1/3}          }{       |x-y|^{4}      }    . \] Let $m\in \N$ be such that $|Q| = 2^{3m}$.   For $x\in Q^*$, we have \EQN{ |I_\text{far} (x,t)| &\leq C|Q|^{1/3} \int_{\R^3\setminus Q^{**}} \frac 1 {|x-y|^4} |u(y,t)|^2\,dy   \\&\leq C\sum_{Q'\in {\mathcal S}_1}|Q|^{1/3} \int_{Q' \cap (Q^{**})^c} \frac 1 {|x-y|^4}  |u|^2\,dy    \\&\indeq   + C\sum_{Q'\in {\mathcal S}_2}|Q|^{1/3} \int_{Q' \cap (Q^{**})^c} \frac 1 {|x-y|^4}  |u|^2\,dy, } where ${\mathcal S}_1$ is the collection of $Q'\in \mathcal C_n$ such that $Q'\subset Q_{m+1}$ while ${\mathcal S}_2$ is the collection of $Q'\in \mathcal C_n$ such that $Q'\subset Q_{m+1}^c$. \par We first focus on $\mathcal S_1$.  The number of cubes $Q'$ in $\mathcal S_1$ is bounded above by  $C \log \bka{|Q|^{1/3}/2^n}$.   Furthermore, $|Q|^{1/3}\sim |x_Q-x_{Q'}|$ whenever $Q'\in \mathcal S_1$.  We claim that if $y\in Q' \cap (Q^{**})^c $ and $x\in Q$, then $|x-y|\sim |x_Q-x_Q'|$. This is clear if $Q'$ and $Q$ are not adjacent. If they are adjacent, then we obviously have $|x_Q-x_{Q'}|\sim 2^m$ and $|x-y|\leq 3 \cdot 2^{m}$. For the lower bound, since $y\notin Q^{**}$, we must have $|x-y|\geq 2^{m-1}$. This proves our claim. Hence, \EQN{   &\sum_{Q'\in {\mathcal S}_1}|Q|^{1/3} \int_{Q' \cap (Q^{**})^c} \frac 1 {|x-y|^4}  |u|^2\,dy     \leq  C \sum_{Q'\in{\mathcal S}_1}         \frac   {|Q|^{1/3}  } {|Q|^{4/3}}  \int_{Q'} |u|^2\,dy    \\&\indeq     \leq  C \sup_{Q'\in {\mathcal S}_1} \frac   {  \log \bka{|Q|^{1/3}/2^n} } {|Q|}  \int_{Q'} |u|^2\,dy \\&\indeq    =  C \sup_{Q'\in {\mathcal S}_1} \frac   {|Q'|^{q/3}\log \bka{|Q|^{1/3}/2^n} } {|Q|} \frac 1 {|Q'|^{q/3}} \int_{Q'} |u|^2\,dy   \\&\indeq    \leq C |Q|^{q/3-1}  \log \bka{|Q|^{1/3}/2^n} \sup_{Q'\in \mathcal C_n} \frac 1 {|Q'|^{q/3}}\int_{Q'} |u|^2 \,dx. } On the other hand, if $Q'\in \mathcal S_2$ then  $Q'\cap (Q^{**})^c = Q'$ (because $Q^{**}$ only overlaps with cubes adjacent to $Q$ which are all subsets of $Q_{m+1}$).  Hence, the $\mathcal S_2$ term is bounded as \EQN{ &\sum_{Q'\in {\mathcal S}_2}|Q|^{1/3} \int_{Q'} \frac 1 {|x-y|^4}  |u|^2\,dy = \sum_{l\geq m+1} \sum_{Q'\in S_l\cap {\mathcal S}_2}\frac {|Q|^{1/3}} {|x_Q-x_{Q'}|^4} \int_{Q'} |u|^2\,dx     \\&\indeq  \leq  C|Q|^{1/3} \sum_{l\geq m+1} \frac {1} {2^{ (4-q) l }}  \sup_{Q'\in \mathcal C_n} \frac 1 {|Q'|^{q/3}}\int_{Q'} |u|^2 \,dx  \leq   \frac {C|Q|^{1/3}} {2^{ (4-q) m }}  \sup_{Q'\in  {\mathcal C}_n} \frac 1 {|Q'|^{q/3}}\int_{Q'} |u|^2 \,dx   \\&\indeq   \leq C |Q|^{q/3-1}    \sup_{Q'\in {\mathcal C}_n} \frac 1 {|Q'|^{q/3}}\int_{Q'} |u|^2 \,dx . }  Combining these estimates yields \eqref{EQL5TS2kIQjZKb9QUx2Ui5Aflw1SLDGIuWUdCPjywVVM2ct8cmgOBS7dQViXR8Fbta1mtEFjTO0kowcK2d6MZiW8PrKPI1sXWJNBcREVY4H5QQGHbplPbwdTxpOI5OQZAKyiix7QeyYI91Ea16rKXKL2ifQXQPdPNL604}, and the proof is concluded. \end{proof} \par \subsection{Main estimate}  In this section we obtain an estimate for solutions in the spaces $M^{2,q}_{\mathcal C_n}$. \par Our estimate for local cubic terms is the following. \par \begin{lemma}\label{lemma.cubic} Let $u\colon\R^3\times (0,T) \to \R^3$. Then, given $\varepsilon>0$,  we have  \EQ{   \frac {1} {|Q|^{1/3}} \int_0^t\int_{Q} |u|^3\,dx\,ds       &\leq C(\varepsilon)   |Q|^{q-4/3} \int_0^t \bigg( \frac 1 {|Q|^{q/3}} \int_{Q}|u|^2 \,dx \bigg)^{3} \,ds  \\&\indeq +  \varepsilon \int_0^t  \int_{Q}|\nb u|^2 \,dx\,ds    \\ &\indeq + C \colb {|Q|^{q/2-5/6}} \int_0^t \bigg( \frac 1 {|Q|^{q/3}} \int_Q |u|^2\,dx \bigg)^{3/2}\,ds,    \label{EQL5TS2kIQjZKb9QUx2Ui5Aflw1SLDGIuWUdCPjywVVM2ct8cmgOBS7dQViXR8Fbta1mtEFjTO0kowcK2d6MZiW8PrKPI1sXWJNBcREVY4H5QQGHbplPbwdTxpOI5OQZAKyiix7QeyYI91Ea16rKXKL2ifQXQPdPNL604} } for any cube $Q\subset \R^3$. \end{lemma}  \par \par \begin{proof} By the Gagliardo-Nirenberg and H\"older inequalities we have \EQN{ &\frac {1} {|Q|^{1/3}} \int_0^t\int_{Q} |u|^3\,dx\,ds  \\&\indeq \leq C\frac {1} {|Q|^{1/3}} \int_0^t \bigg(  \int_{Q}|u|^2 \,dx \bigg)^{3/4} \bigg(   \int_{Q}|\nb u|^2 \,dx \bigg)^{3/4}\,ds \\&\indeq\indeq + C\frac {1} {|Q|^{1/3}} \int_0^t \frac 1 {|Q|^{1/2}} \bigg( \int_Q |u|^2\,dx \bigg)^{3/2}ds \\&\indeq \leq C\frac {1} {|Q|^{1/3}}  \bigg(  \int_0^t \bigg(  \int_{Q}|u|^2 \,dx \bigg)^{3} \,ds \bigg)^{1/4} \bigg( \int_0^t  \int_{Q}|\nb u|^2 \,dx\,ds \bigg)^{3/4} \\&\indeq\indeq+ C\frac {  |Q|^{q/2}} {|Q|^{5/6}} \int_0^t \bigg( \frac 1 {|Q|^{q/3}} \int_Q |u|^2\,dx \bigg)^{3/2}\,ds. } Using Young's inequality $ab \lec a^4+b^{4/3}$, the last expression is bounded by \EQN{ &\indeq \leq  C \frac {   |Q|^{q/ 4}} {|Q|^{1/3}}  \bigg(  \int_0^t \bigg(  \frac 1 {|Q|^{q/3}} \int_{Q}|u|^2 \,dx \bigg)^{3} \,ds \bigg)^{1/4} \bigg(  \int_0^t  \int_{Q}|\nb u|^2 \,dx\,ds \bigg)^{3/4}  \\&\indeq\indeq  + C    {|Q|^{q/2-5/6}} \int_0^t \bigg( \frac 1 {|Q|^{q/3}} \int_Q |u|^2\,dx \bigg)^{3/2}\,ds \\&\indeq\leq C(\varepsilon ) |Q|^{ q-4/3} \int_0^t\bigg( \frac 1 {|Q|^{q/3}} \int_{Q}|u|^2 \,dx \bigg)^{3} \,ds  +  \varepsilon \int_0^t  \int_{Q}|\nb u|^2 \,dx\,ds \\&\indeq\indeq  + C   {|Q|^{q/2-5/6}} \int_0^t \bigg( \frac 1 {|Q|^{q/3}} \int_Q |u|^2\,dx \bigg)^{3/2}\,ds, } and the proof is concluded. \end{proof} \par \par We study the local energy by way of cutoff functions  $\phi_{Q}$ which we now define.  Let $\phi$ be a radial smooth cutoff function such that $\phi= 1$ in $[-1/2,1/2]^3$ and $\phi = 0$ off of $[-2/3,2/3]^3$ with $\phi$ non-increasing in $|x|$. For $Q\in \mathcal C_n$, let $\phi_Q$ be the translation and dilation of $\phi$ so that $\phi_Q$ equals $1$ on $Q$ and vanishes off of $Q^*$.  Then, $\|\partial^\la \phi_Q(x)\|_{L^\I}\leq C(\lambda) / |Q|^{|\la|/3}$ where $C$ does not depend on $Q$ and $\la$ is any multi-index.   \par \par We need the following version of the Gr\"onwall inequality. \par \begin{lemma}\label{lemma.gronwall} Suppose $f(t) \in L^\infty_\loc([0,T); [0,\infty))$ satisfies, for some $m \ge 1$,   \[ f(t) \le a + \int_0^t( b_1f(s) +b_2 f(s)^m) ds    \comma 0<t<T,   \] where $a,b>0$, then for $T_0=\min(T,T_1)$, with   \begin{equation*}    T_1  =  \frac a { b_1 2a + b_2(2a)^m},   \end{equation*} we have $f(t) \le 2a$ for $t \in (0,T_0)$. \end{lemma}  The proof is obtained by a barrier argument; cf.~\cite{BT8}) for details. \par \par Denote \[ \al_{n}(t) =   \| u(t)\|_{  M^{2,2}_{\mathcal C_n}}^2 \] and \[ \be_{n}(t)=\sup_{Q\in \mathcal C_n} \frac 1 {|Q|^{2/3}}\int_0^t\int_{Q}|\nb u(x,s)|^2\,dx\,ds. \] \par \par We now state the bound on the existence times.   It is motivated by an estimate in \cite{JiaSverak-minimal}. \par \begin{theorem}\label{thrm.main2}Assume $u_0\in M^{2,2}_{\mathcal C_n}$,  is divergence-free, and  let $(u,p)$ be a local energy solution with initial data $u_0$ on $\R^3\times (0,\I)$.  Assume additionally that   \EQN{ \esssup_{0<s<t}\al_{n }(s)+\be_{n}(t)<\I, } for all $t<\I$. Then there exist  constants $c_0$ and $c_1$ both independent of $n$ so that,  setting \EQ{ {T_n = \frac {c_1} {  2^{-2n}+ \|u_0\|^4_{M^{2,2}_{\mathcal C_n}}  }, }    \label{EQL5TS2kIQjZKb9QUx2Ui5Aflw1SLDGIuWUdCPjywVVM2ct8cmgOBS7dQViXR8Fbta1mtEFjTO0kowcK2d6MZiW8PrKPI1sXWJNBcREVY4H5QQGHbplPbwdTxpOI5OQZAKyiix7QeyYI91Ea16rKXKL2ifQXQPdPNL607} } we have  \[ \sup_{0<s<T_n }\|u(s)\|_{M^{2,2}_{\mathcal C_n}}^2  + \sup_{Q\in \mathcal C_n} \frac 1 {|Q|^{2/3}}\int_0^{T_n}\int_{Q}|\nb u|^2\,dx\,ds \leq 2c_0 \| u_0\|_{M^{2,2}_{\mathcal C_n}}^2. \] \end{theorem} \par Let us briefly explain why this bound is useful. In comparison to the analogous estimate in \cite{BK1}, we have control over the time scales via the parameter $n$.  Indeed, if $\|u_0\|_{M^{2,2}_{\mathcal C_n}}\to 0$ as $n\to \I$, then $T_n\to \I$.  This leads to a global solution in the next section. \par \par \par \par \begin{proof}[Proof of Theorem~\ref{thrm.main2}] Fix $Q\in \mathcal C_n$.  The local energy inequality and  the item 5 of Definition~\ref{def:localenergy} give \EQN{ &\frac12\int  {|u(x,t)|^2}  \phi_Q(x)\,dx + \int_0^t\int |\nb u(x,s)|^2\phi_Q(x)\,dx\,ds  \\&\indeq\leq \frac12\int  {|u(x,0)|^2}  \phi_Q(x)\,dx     +\frac12\int_0^t\int  { |u(x,s)|^2}  \Delta \phi_Q (x)\,dx\,ds \\&\indeq\indeq   + \frac12\int_0^t \int \Bigl( {|u(x,s)|^2}                        u\cdot \nb \phi_Q(x) 		     +  2  (p(x,s) - p_Q(s)) u(x,s)\cdot\nb \phi_Q(x)                                     \Bigr)\,dx\,ds. } For the linear term we have, \EQN{ &\frac12\int_0^t \int  |u|^2  \Delta \phi_Q \,dx\,ds  \leq  \frac {C  } {|Q|^{ 2/3}} \int_0^t\sup_{Q'\cap Q^*\neq \emptyset} \int_{Q'} {|u|^2}  \,dx \,ds \\&\indeq \leq C  \int_0^t  \sup_{Q'\cap Q^*\neq \emptyset} \frac 1 {|Q'|^{q/3}} \frac { |Q'|^{q/3} }{|Q|^{ 2/3}} \int_{Q'} {|u|^2}  \,dx \,ds \leq C\,|Q|^{(q-2)/3}\int_0^t     \al_{n}(s)  \,ds, } where we used   the fact that $|Q'|$ and $|Q|$ are comparable.   \par The remaining terms are bounded using Lemmas~\ref{lemma.pressure} and~\ref{lemma.cubic}.  This leads to  \EQ{\label{ineq.3.7} & \int  {|u(x,t)|^2}  \phi_Q(x)\,dx +2 \int_0^t\int |\nb u(x,s)|^2\phi_Q(x)\,dx\,ds   \\&\indeq \leq \int  {|u(x,0)|^2}  \phi_Q(x)\,dx  +   C|Q|^{(q-2)/3}\int_0^t  \al_n(s)\,ds  \\& \indeq\indeq + C (\log \bka{|Q|^{1/3}/2^n})^{3/2}  |Q|^{q/2-5/6} \int_0^t \al_n(s)^{3/2} \,ds  \\& \indeq\indeq +C |Q|^{ q-4/3}   \int_0^t \al_n(s)^{3} \,ds   \\&\indeq\indeq + \varepsilon \sup_{Q'\in \mathcal C_n;Q'\cap Q^{**}\neq \emptyset} |Q'|^{q/2-1/3} \frac 1 { |Q'|^{q/3}}\int_0^t  \int_{Q'}|\nb u|^2 \,dx\,ds, }  where $\varepsilon>0$ is fixed (with the choice explained after \eqref{ineq.3.10} below); the constants between \eqref{ineq.3.7} and \eqref{ineq.3.10} depend on $\varepsilon$. Next, plug in $q=2$ and divide through  by $|Q|^{2/3}$. Then, by Young's inequality,  \EQN{ &C (\log \bka{|Q|^{1/3}/2^n})^{3/2} |Q|^{q/2-5/6-2/3} \int_0^t \al_n(s)^{3/2} \,ds  \\&\indeq \leq C  (\log \bka{|Q|^{1/3}/2^n})^{2} |Q|^{(1-5/6-2/3)4/3} \int_0^t  \al_n(s)\,ds    + \int_0^t\al_n(s)^{3} \,ds  \\&\indeq = C \frac {(\log \bka{|Q|^{1/3}/2^n})^2} {(|Q|^{1/3} /2^n)^22^{2n}}  \int_0^t \al_n(s)\,ds + \int_0^t\al_n(s)^{3} \,ds. } For $Q\in \mathcal C_n$, \[  \frac {(\log \bka{|Q|^{1/3}/2^n})^2} {(|Q|^{1/3} /2^n)^2 } \] is bounded and the bound is independent of $n$. Hence, \EQN{ &C  (\log \bka{|Q|^{1/3}/2^n})^{3/2} |Q|^{q/2-5/6-2/3} \int_0^t \al_n(s)^{3/2} \,ds \\&\indeq   \leq  C 2^{-2n}    \int_0^t  \al_n(s) \,ds + \int_0^t \al_n(s)^{3} \,ds. } Using this in \eqref{ineq.3.7} divided by $|Q|^{2/3}$ leads to \EQ{\label{ineq.3.10} &\frac 1 {|Q|^{2/3}} \int  {|u(x,t)|^2}  \phi_Q(x)\,dx + \frac 2 {|Q|^{2/3}} \int_0^t\int |\nb u(x,s)|^2\phi_Q(x)\,dx\,ds   \\&\indeq \leq\frac 1 {|Q|^{2/3}}  \int  {|u(x,0)|^2}  \phi_Q(x)\,dx  +    \frac C {2^{2n}}\,\int_0^t \al_n(s)\,ds   +C  \int_0^t\al_n(s)^{3} \,ds   \\&\indeq\indeq +  \sup_{Q'\in \mathcal C_n;Q'\cap Q^{**}\neq \emptyset} \frac 1 { |Q'|^{2/3}}\int_0^t  \int_{Q'}|\nb u|^2 \,dx\,ds, } where, to obtain the last line, $\varepsilon$ is chosen to be small enough so that $\varepsilon |Q'|^{2/3}  / |Q|^{2/3} \leq 1$ whenever $Q'\cap Q^{**}\neq \emptyset$. Such $\varepsilon$ can be chosen independently of $Q$ and $n$.  Taking the supremum over  $Q\in \mathcal C_n$ and absorbing the gradient term in the left hand side gives \EQN{ \| u(t)\|_{M^{2,2}_{\mathcal C_n}}^2 \leq  c_0 \| u_0\|_{M^{2,2}_{\mathcal C_n}}^2   + C  {2^{-2n}}  \int_0^t \| u(s)\|_{M^{2,2}_{\mathcal C_n}}^2\,ds + C   \int_0^t \| u(s)\|_{M^{2,2}_{\mathcal C_n}}^6\,ds, } where $c_0$ is a fixed constant depending only on $\phi$ chosen to ensure \[ \sup_{Q\in \mathcal C_n} \frac 1 {|Q|^{2/3}} \int |u_0|^2\phi_Q \,dx \leq c_0  \| u_0 \|_{M^{2,2}_{\mathcal C_n}}^2. \] By Lemma~\ref{lemma.gronwall} we obtain \[ \| u(t)\|_{M^{2,2}_{\mathcal C_n}}^2 \leq 2c_0 \|u_0\|_{M^{2,2}_{\mathcal C_n}}^2, \] provided   \begin{equation}    t\leq \frac {C   } {     2^{-2n} + \|u_0\|_{M^{2,2}_{\mathcal C_n}}^4}       .    \notag     \end{equation} Thus letting $c_1$ be the constant from the above inequality completes the proof. \end{proof} \par  \begin{remark}\label{remark.1} Theorem~\ref{thrm.main2} can be re-formulated for $M^{2,q}_{\mathcal C_n}$ for $0\leq q<2$. To prove this, just start at \eqref{ineq.3.7} and proceed \labell{xGc3tvKLXaC1dKgw9H3o2kEoulIn9TSPyL2HXO7tSZse01Z9HdslDq0tmSOAVqtA1FQzEMKSbakznw839wnH1DpCjGIk5X3B6S6UI7HIgAaf9EV33Bkkuo3FyEi8Ty2ABPYzSWjPj5tYZETYzg6Ix5tATPMdlGke} with a different choice of $q$. Doing so for $q=1$ leads to the estimate \EQ{ \notag \sup_{0<s<\bar T_n }\|u(s)\|_{M^{2,1}_{\mathcal C_n}}^2  + \sup_{Q\in \mathcal C_n} \frac 1 {|Q|^{1/3}}\int_0^{\bar T_n}\int_{Q}|\nb u|^2\,dx\,ds \leq 2c_0 \| u_0\|_{M^{2,1}_{\mathcal C_n}}^2, } where  \EQ{ \bar T_n \sim \frac  1 {2^{-2n} + 2^{-2n} \|u_0\|_{M^{2,1}_{\mathcal C_n}}^4}.    \label{EQL5TS2kIQjZKb9QUx2Ui5Aflw1SLDGIuWUdCPjywVVM2ct8cmgOBS7dQViXR8Fbta1mtEFjTO0kowcK2d6MZiW8PrKPI1sXWJNBcREVY4H5QQGHbplPbwdTxpOI5OQZAKyiix7QeyYI91Ea16rKXKL2ifQXQPdPNL606} } If additionally $u_0\in \mathring M^{2,1}_{\mathcal C}$, then $\|u_0\|_{M^{2,1}_{\mathcal C_n}}\to 0$ as $n\to \I$ and we may take  $\bar T_n \sim 2^{2n} $ for large $n$. \end{remark} \par \section{Construction of global solutions}\label{sec.approximation} \par Recall that the initial data can be approximated by the following lemma proven in \cite{BK1}. The second conclusion is new and follows from observations in Section~\ref{sec.data}. \begin{lemma}\label{lemma.approx} Assume $f\in \MC$ is divergence-free.  For every $\e>0$ there exists  a divergence-free $g\in L^2$ such that $\| f-g\|_{\MM} \leq\e$ and $\| f-g\|_{M^{2,2}_{\mathcal C_n}}<\e$ for every $n\in \N$ ($g$ is independent of $n$). \end{lemma} \par Our proof follows the global existence argument in \cite{BT8} very closely. The basic elements of this argument were first written down in  \cite{LR} and later elaborated in \cite{KiSe}.   See also \cite{LR-Morrey, KwTs,LR2,MaMiPr}. \par \begin{proof}  \par We argue by induction.   For $n\in \N$, let $u_0^n$ be as in Lemma~\ref{lemma.approx} where $\|u_0 - u_0^n\|_{M^{2,2}_\mathcal C}<\frac 1 n$ and $\|u_0 - u_0^n\|_{M^{2,2}_{\mathcal C_k}}<1/n$ for all $k \in \N$. Let $u^n$ and $\td p^n$ be a global solution in the Leray class with $u^n(0)=u_0^n$. By the classical theory (see \cite{Tsai} for a reference) we may furthermore assume that $u^n$ is suitable and $\td p^n$ satisfies the local pressure expansion.   Let $T_n = \inf_{j\geq n} \bar T_j$ where $\bar T_j$ is the time-scale from Theorem~\ref{thrm.main2}.   This sequence is non-decreasing and $T_n\to \I$ as $n\to \I$. Let $B_n$ denote the ball centered at the origin of radius $n$.  Then, Theorem~\ref{thrm.main2} implies that $u^n$ are uniformly bounded in the class from inequalities \cite[(4.1)--(4.4)]{KiSe} on $B_1\times [0,T_1]$.  Hence, there exists a sub-sequence  $u^{1,k}$ that converges to a vector field $u_1$  on $B_1\times (0,T_1)$ in the following sense \begin{align*} &u^{1,k}\overset{\ast}\rightharpoonup u_1\quad \text{in }L^\I(0,T_1;L^2(B_1)) \\&u^{1,k}  \rightharpoonup u_1\quad \text{in }L^2(0,T_1;H^1(B_1)) \\&u^{1,k} \to u_1  \quad \text{in }L^3(0,T_1;L^3(B_1)). \end{align*} \par By Theorem~\ref{thrm.main2}, all $u^{1,k}$ are also uniformly bounded on $B_n\times [0,T_n]$ for $n\in \N$, $n \ge2$.  Therefore, we can inductively extract  subsequences $\{u^{n,k}\}_{k \in \N}$ from $\{u^{n-1,k}\}_{k \in \N}$ which converge to  a vector field $u_n$ on $B_n\times (0,T_n)$ as $k \to \infty$ in the following sense \begin{align*} &u^{n,k} \overset{\ast}\rightharpoonup u_n\quad \text{in }L^\I(0,T_n;L^2(B_n)) \\&u^{n,k}  \rightharpoonup u_n\quad \text{in }L^2(0,T_n;H^1(B_n)) \\&u^{n,k} \to u_n \quad \text{in }L^3(0,T_n;L^3(B_n)). \end{align*} Let $\tilde u_n$ be the extension by $0$ of $u_n$ to $\R^3\times (0,\I)$. Note that, at each step, $\tilde u_n$ agrees with $\tilde u_{n-1}$ on $B_{n-1}\times (0, T_{n-1})$.  Let $u=\lim_{n\to \I}\tilde u_n$. Then, $u=u_n$ on $B_n\times (0,T_n)$ for every $n\in \N$. \par \par Let $u^{(k) }= u^{k,k}$ on $B_k\times (0,T_k)$ and equal $0$ elsewhere.    Then, for every fixed $n$ and as $k \to \infty$, \EQN{ &u^{(k) } \overset{\ast}\rightharpoonup u\quad \text{in }L^\I(0,T_n;L^2(B_n)) \\&u^{(k) }  \rightharpoonup u\quad \text{in }L^2(0,T_n;H^1(B_n)) \\&u^{(k) } \to u \quad \text{in }L^3(0,T_n;L^3(B_n)) . } \par \par Based on the uniform bounds for the approximates, we have that $u$ satisfies \EQN{ &\sup_{0<t\leq T_n} \sup_{x_0\in \R^3} \int_{B_n(x_0)}|u(x,t)|^2\,dx  \\&\indeq + \sup_{x_0\in \R^3}\int_0^{T_n} \int_{B_n(x_0)} |\nb u(x,t)|^2\,dx\,ds \leq  C \sup_{x_0\in \R^3} \int_{B_n(x_0)} |u_0|^2\,dx. } \par  The pressure is dealt with as in \cite[\S3]{KwTs}. Let \EQN{ p^{(k) } (x,t) =& -\frac 1 3  |u^{(k) }|^2 (x,t) + \pv \int_{B_2} K_{ij}(x-y)  (u^{(k) }_i   \, u^{(k) }_j )(y,t) \,dy   \\&+ \pv \int_{B_2^c} (K_{ij}(x-y)-K_{ij}(-y))   (u^{(k) }_i  \, u^{(k) }_j )(y,t) \,dy,  } which differs from the pressure   associated to $u^{(k) }$  by a function of $t$ which is constant in $x$, and so $u^{(k) }$ with the above pressure $p^{(k) }$ is also a distributional solution to \eqref{eq.NSE}. \par Since $u^{(k)}$ converges to $u$ in the above sense, it follows that $p^{(k) }\to p$ in $L^{3/2}(0,T_m;L^{3/2} (B_m) )$ for all $m$ and $p$   given by \EQN{ p(x,t)  = \lim_{m \to \infty} \bar p^m(x,t) , } where $\bar p^m(x,t)$ is defined on $B_{2^m}(0)$   by  \EQN{ \bar p^m(x,t)  =&-\frac 13 |u(x,t)|^2  +\pv \int_{B_2}  K_{ij}(x-y) u_i u_j(y,t) dy + \bar p^m_3+\bar p^m_4, } with \EQN{ \bar p^m_3(x,t) &=\pv \int_{B_{2^{m+1}}\setminus B_2}  (K_{ij}(x-y)-K_{ij}(-y)) u_i u_j(y,t)\,  dy , \\ \bar p^m_4(x,t) &=\int_{B_{2^{m+1}}^c}  (K_{ij}(x-y)-K_{ij}(-y)) u_i u_j(y,t) \, dy. }  To prove this one adapts the argument concerning the convergence of the pressure in \cite[Proof of Theorem~1.3]{BK1}.  \par \par We now establish the local pressure expansion. Following the analogous argument in \cite{BK1}, it is possible to \labell{67Xb7FktEszyFycmVhGJZ29aPgzkYj4cErHCdP7XFHUO9zoy4AZaiSROpIn0tp7kZzUVHQtm3ip3xEd41By72uxIiY8BCLbOYGoLDwpjuza6iPakZdhaD3xSXyjpdOwoqQqJl6RFglOtX67nm7s1lZJmGUrdIdXQ} show that \EQ{\label{convergence} G_{ij}^Q(u_i^{(k)} u_i^{(k)} ) \to G_{ij}^Q (u_i u_j)\quad \text{ in }L^{3/2}(0,T_;L^{3/2}(Q) ), } for any cube  $Q$ and time $T>0$. Fix a cube $Q$  and $T>0$. Using \eqref{convergence} and taking the limit of the weak form of \eqref{eq.NSE} we find that the pair $ (u, G_{ij}^Q (u_i u_j)) $ solves \eqref{eq.NSE} in $Q\times (0,T)$. Hence, $\nb p = \nb  G_{ij}^Q (u_i u_j))$ in $\mathcal D'(\R^3)$ at every time $t$ and so there exists a constant $p_Q(t)$ so that  \[ p (x,t) = G_{ij}^Q (u_i u_j))(x,t) + p_Q(t), \] for $x\in Q$.  Clearly we have $p_Q\in L^{3/2}(0,T)$. This gives the desired local pressure expansion. \par At this point we have established the items 1, 2, and 6 from the definition of local energy solutions. The remaining items follow from the arguments  in \cite{BK1} (see also \cite[pp.~156--158]{KiSe} and \cite[\S3]{KwTs}). This is because for any time $T_0$, we have the same convergences of $u^k$ and $p^k$ on $B_n\times T_0$ for all $n\in \N$ as in \cite{BK1}; cf.~also \cite{BT8}.   \end{proof} \par \section{Eventual regularity}\label{sec.er} \par In this section we use Remark~\ref{remark.1} to prove Theorem~\ref{thrm.er},  but first recall a variant of the Caffarelli-Kohn-Nirenberg regularity  criteria \cite{CKN} due to Lin \cite{L98}; see also \cite{LS99,Ku}. \par   \begin{lemma}[$\e$-regularity criteria]\label{thrm.epsilonreg}   For any $\si\in (0,1)$,  there exists a universal constant $\e_*=\e_*(\si)>0$ such that, if a pair $(u,p)$ is a suitable weak solutions of \eqref{eq.NSE} in   $Z_r=Z_r(x_0,t_0)=B_r(x_0)\times (t_0-r^2,t_0)$,  and \[ \e^3=\frac 1 {r^2} \int_{Z_r} (|u|^3 +|p|^{3/2})\,dx\,dt <\e_*, \] then $u\in L^\I(Z_{\si r})$. Moreover, \[ \|  \nabla^k u\|_{L^\I(Z_{\si r})} \leq C_k {\e}\, r^{-k-1}    \comma k\in {\mathbb N}_0 \] for universal constants $C_k=C_k(\si)$. \end{lemma} \par The usual version is $\si=1/2$. The general version $\si \in (0,1)$ follows from the usual version and a partition argument. \par \par \par \begin{proof}[Proof of Theorem~\ref{thrm.er}]   Choose $N_1$ such that  $\| u_0\|^2_{M^{2,1}_{\mathcal C_n}} <1$ for all $n\geq N_1$.  Let    \begin{equation}    Q= Q_{n-1} = (-2^n,2^n)^3  .     \notag   \end{equation} Note that $Q\in \mathcal C_n$ by \eqref{EQL5TS2kIQjZKb9QUx2Ui5Aflw1SLDGIuWUdCPjywVVM2ct8cmgOBS7dQViXR8Fbta1mtEFjTO0kowcK2d6MZiW8PrKPI1sXWJNBcREVY4H5QQGHbplPbwdTxpOI5OQZAKyiix7QeyYI91Ea16rKXKL2ifQXQPdPNL611}. By Lemmas~\ref{lemma.pressure} and removing the log factor due to $|Q|^{1/3}= 2^{n+1}$,    \EQ{   \label{EQL5TS2kIQjZKb9QUx2Ui5Aflw1SLDGIuWUdCPjywVVM2ct8cmgOBS7dQViXR8Fbta1mtEFjTO0kowcK2d6MZiW8PrKPI1sXWJNBcREVY4H5QQGHbplPbwdTxpOI5OQZAKyiix7QeyYI91Ea16rKXKL2ifQXQPdPNL613}       J&=   \frac{1}{|Q|^{2/3}}    \int_0^{t} \int_{Q} (|u|^3+|p- p_{Q}(s)|^{3/2}  )\,dx\,ds    \\&    \leq       \sup_{Q'\in \mathcal C_n ; Q'\cap Q^{**}\neq \emptyset} \frac C {|Q'|^{2/3}} \int_0^t \int_{Q'} |u|^3\,dx\,ds      \\& \indeq\indeq     + C |Q|^{\frac q2-\frac76}\int_0^t  \bigg(\sup_{Q'\in \mathcal C_n}\frac 1 {|Q'|^{q/3}} \int_{Q'} |u(x,s)|^2\,dx    \bigg)^{3/2} \,ds .     } Using Lemma~\ref{lemma.cubic} with $\varepsilon = 1$, \EQN{    J &    \leq     \sup_{Q'\in \mathcal C_n ; Q'\cap Q^{**}\neq \emptyset}   |Q|^{q-5/3} \int_0^t \bigg( \frac 1 {|Q|^{q/3}} \int_{Q}|u|^2 \,dx \bigg)^{3} \,ds      \\&\indeq\indeq        + |Q|^{-1/3}\sup_{Q'\in \mathcal C_n ; Q'\cap Q^{**}\neq \emptyset}            \int_0^t  \int_{Q}|\nb u|^2 \,dx\,ds    \\ &\indeq \indeq       + C\sup_{Q'\in \mathcal C_n ; Q'\cap Q^{**}\neq \emptyset}            {|Q|^{\frac q2-\frac76}} \int_0^t \bigg( \frac 1 {|Q|^{q/3}} \int_Q |u|^2\,dx \bigg)^{3/2}\,ds     \\&\indeq\indeq     + C |Q|^{\frac q2-\frac76}      \int_0^t           \bigg(\sup_{Q'\in \mathcal C_n}\frac 1 {|Q'|^{q/3}} \int_{Q'} |u(x,s)|^2\,dx    \bigg)^{3/2} \,ds. } Since the first \labell{7jps7rcdACYZMsBKANxtkqfNhktsbBf2OBNZ5pfoqSXtd3cHFLNtLgRoHrnNlwRnylZNWVNfHvOB1nUAyjtxTWW4oCqPRtuVuanMkLvqbxpNi0xYnOkcdFBdrw1Nu7cKybLjCF7P4dxj0Sbz9faVCWkVFos9t2aQ} three terms are the supremum over a finite number (independent of $n$) of cubes, we may move the supremum inside the integrals. Taking $q=1$, we get \EQN{     J&    \leq   \frac{C}{|Q|^{2/3}}             \int_0^t              \bigg(  \sup_{Q'\in \mathcal C_n  }           \frac 1 {|Q'|^{1/3}} \int_{Q'}|u|^2 \,dx          \bigg)^{3} \,ds          +   \frac{C}{|Q|^{2/3}}   \int_0^t \bigg( \sup_{Q'\in \mathcal C_n  }           \frac 1 {|Q'|^{1/3}} \int_{Q'} |u|^2\,dx \bigg)^{3/2}\,ds       \\ &\indeq\indeq           +  \frac{C}{|Q|^{1/3}}\sup_{Q'\in \mathcal C_n  }           \int_0^t  \int_{Q'}|\nb u|^2 \,dx\,ds. } Then, using  Remark \ref{remark.1},  $\| u_0\|^2_{M^{2,1}_{\mathcal C_n} }\leq 1$, and taking $t=\bar T_n = c_*2^{2n} $ where $c_*\in (0,1]$ is the suppressed constant in~\eqref{EQL5TS2kIQjZKb9QUx2Ui5Aflw1SLDGIuWUdCPjywVVM2ct8cmgOBS7dQViXR8Fbta1mtEFjTO0kowcK2d6MZiW8PrKPI1sXWJNBcREVY4H5QQGHbplPbwdTxpOI5OQZAKyiix7QeyYI91Ea16rKXKL2ifQXQPdPNL606}, we obtain \EQ{\label{Eq5.2} \frac 1 {|Q|^{2/3}} \int_0^{\bar T_n} \int_{Q} (|u|^3{+} |p- p_{Q}(s)|^{3/2}  )\,dx\,ds \leq  C \| u_0\|^2_{M^{2,1}_{\mathcal C_n} } } for  $n\geq N_1$. \par We now prove Theorem ~\ref{thrm.er}. We may assume $0<\delta<1$. Define $\si\in (0,1)$ by \EQ{\label{Eq5.3} \si^2 = \left(1+\frac{\delta}{4}\right)^{-1}\in \left(\frac45,1\right). } By $c_*\leq 1$, we have \[ B_{\sqrt{c_*}2^n}(0)\times [0,c_* 2^{2n}] \subset Q_{n-1}\times [0,c_*2^{2n}], \] for large $n$.  Since $\| u_0\|^2_{M^{2,1}_{\mathcal C_n}} \to 0$, by \eqref{Eq5.2}, there exists $N_2 \geq N_1$ so that if $n\geq N_2$, then \[ \frac 1 {c_*2^{2n}} \int_0^{ c_*2^{2n}} \int_{B_{\sqrt{c_*}2^n }(0)} (|u|^3 {+} |p- p_Q|^{3/2}  )\,dx\,dt  < \frac{ \e_*(\si)}{4}. \] By Lemma~\ref{thrm.epsilonreg}, $u$ is regular in  \[ Z_n = B_{\si \sqrt{c_*}2^n }(0) \times [(1-\si^2) c_*2^{2n} , c_*2^{2n}] \] and $\norm{u}_{L^\infty(Z_n)} \le C 2^{-n}$.  Note that $Z_n$ contains, using \eqref{Eq5.3}, \[ P_n = \bket{(x,t) \in \R^{3+1}_+ : \delta|x|^2 \le t, \quad (1-\si^2) c_*2^{2n}  \le t\le  4(1-\si^2) c_*2^{2n}}. \] Hence, $u$ is regular in \[ \bigcup _{n \ge N_2} P_n= \bket{(x,t) \in \R^{3+1}_+ : \max\{\tau, \delta|x|^2\}\le t}    \comma  \tau =  (1-\si^2) c_*2^{2N_2}. \] This proves Theorem~\ref{thrm.er}.  \end{proof} \par \par \par \begin{remark} In order to apply Lemma~\ref{thrm.epsilonreg}, we hope to bound both terms on the right side of \eqref{EQL5TS2kIQjZKb9QUx2Ui5Aflw1SLDGIuWUdCPjywVVM2ct8cmgOBS7dQViXR8Fbta1mtEFjTO0kowcK2d6MZiW8PrKPI1sXWJNBcREVY4H5QQGHbplPbwdTxpOI5OQZAKyiix7QeyYI91Ea16rKXKL2ifQXQPdPNL613} by $\norm{u_0}_{M^{2,q}_{\mathcal C_n}}^2$. Anticipating $t \sim |Q|^{2/3}$, we need for the second term \[ \frac q2-\frac76 + \frac 23 \le 0,  \] i.e., $ q\le 1$. Thus $q=1$ is the largest $q$ we may choose to obtain the eventual regularity with our method. \end{remark} \par \par \section*{Acknowledgments}  ZB was supported in part by the Simons Foundation, IK was supported in part by the NSF grant DMS-1907992, while TT was supported in part by NSERC grant 261356-18.   \par \par \par \par \par  \end{document}